\pdfoutput=1 
\documentclass[10pt]{amsart}

\title[Quasi-Isomorphisms and Divided Power Structures]{Quasi-Isomorphisms of
Commutative DG Rings and Divided Power
Structures}
\usepackage{hyperref}
\hypersetup{colorlinks=false}
\author{Amnon Yekutieli}
\address{Department of Mathematics,
Ben Gurion University, Be'er Sheva 84105, Israel.
\newline \indent \textup{\textit{Email}:
\href{mailto:amyekut@gmail.com}{\rm \scriptsize
\nolinkurl{amyekut@gmail.com}},
\textit{Web}: \rm \scriptsize
\url{https://sites.google.com/view/amyekut-math/home}}}
\date{21 Oct 2023, {\em Version} 2}

\usepackage[theoremfont]{newtxtext}
\usepackage[scaled=0.92]{roboto}  
\usepackage[varbb, varvw, smallerops]{newtxmath}
\usepackage[scaled=0.90]{inconsolata}
\usepackage[cal=cm, calscaled=0.95, frak=euler, frakscaled=0.98]{mathalfa}
\usepackage[T1]{fontenc}

\usepackage{tikz}
\usepackage{tikz-cd}
\usetikzlibrary{arrows}
\tikzcdset{arrow style=tikz,
diagrams={>={Stealth[round,length=4pt,width=6pt,inset=3pt]}}}
\tikzcdset{every label/.append style = {font = \footnotesize}}
\usepackage{graphicx}
\usepackage{hyperref}

\newtheorem{thm}[equation]{Theorem}

\newtheorem{lem}[equation]{Lemma}
\theoremstyle{definition}

\newtheorem{rem}[equation]{Remark}

\newcommand{\xar}{\xrightarrow}

\newcommand{\sub}{\subseteq}
\newcommand{\opn}{\operatorname}
\newcommand{\cat}[1]{\operatorname{\mathsf{#1}}}

\newcommand{\cd}{\mspace{1.8mu}{\cdotB}\mspace{2.0mu}}
\newcommand{\pl}{\msp{1.0} + \msp{1.0}}

\newcommand{\rmitem}[1]{\item[\text{\textup{(#1)}}]}

\newcommand{\mrm}[1]{\mathrm{#1}}

\newcommand{\ga}{\gamma}

\newcommand{\K}{\mathbb{K}}

\newcommand{\Q}{\mathbb{Q}}
\newcommand{\Z}{\mathbb{Z}}

\newcommand{\tup}[1]{\textup{#1}}

\newcommand{\boplus}{\bigoplus\nolimits}
\newcommand{\ot}{\otimes}

\newcommand{\til}[1]{\tilde{#1}}

\renewcommand{\d}{\mathrm{d}}

\newcommand{\lb}{\linebreak}

\newcommand{\ov}{\mspace{0mu} / \mspace{-0.5mu}}

\newcommand{\lmsp}{\mspace{1.5mu}}
\newcommand{\mmsp}{\mspace{3mu}}

\newcommand{\msp}[1]{\mspace{#1 mu}}

\begin{document}

\begin{abstract}
We prove that a quasi-isomorphism $f : A \to B$ between commutative DG rings,
where $B$ admits a divided power structure, can be factored as
$f = \til{f} \circ e$, where
$e : A \to \til{B}$ is a split injective quasi-isomorphism, and
$\til{f} : \til{B} \to B$ is a surjective quasi-isomorphism.

This result is used in our work on a DG approach to the cotangent complex, and
our work on the derived category of commutative DG rings.
\end{abstract}

\maketitle

The purpose of this short paper is to prove the following theorem:

\begin{thm} \label{thm:100}
Let $f : A \to B$ be a quasi-isomorphism of CDG rings. Assume that
$B$ admits a divided power structure. Then there exists a CDG ring $\til{B}$,
with CDG ring homomorphisms $e : A \to \til{B}$, $p : \til{B} \to A$ and
$\til{f} : \til{B} \to B$, such that
$p \circ e = \opn{id}_A$, $\til{f} \circ e = f$, and $\til{f}$ is a surjective
quasi-isomorphism.
\end{thm}

It is easy to see that $p$ and $e$ are quasi-isomorphisms too.
So $e : A \to \til{B}$ is
a split injective quasi-isomorphism, and $p : \til{B} \to A$ is surjective
quasi-isomorphism.
Here is the commutative diagram, in the category $\cat{CDGRng}$
of commutative DG rings, illustrating the theorem.
\begin{equation} \label{eqn:100}
\begin{tikzcd} [column sep = 14ex, row sep = 7ex]
A
&
\til{B}
\ar[d, two heads, "{\til{f}, \mmsp \mrm{sqi}}"]
\ar[l, two heads, "{p, \mmsp \mrm{sqi}}"']
\\
A
\ar[u, two heads, tail, "{\opn{id}_A, \mmsp \mrm{isom}}"]
\ar[ur, tail, "{ \quad e, \mmsp \mrm{qi}}" inner sep = 0.2ex]
\ar[r, "{f, \mmsp \mrm{qi}}"']
&
B
\end{tikzcd}
\end{equation}

This is the organization of the paper: we start by explaining the
terminology and recalling some relevant material; then we state and prove two
lemmas; and finally we prove Theorem \ref{thm:100}.

A {\em commutative graded ring} is a nonpositive graded ring
$A = \bigoplus_{i \leq 0} A^i$,
satisfying the {\em strong commutativity conditions}
$b \cd a = (-1)^{i \cd j} \cd a \cd b$
for all $a \in A^i$ and $b \in A^j$, and $a \cd a = 0$ when $i$ is odd.
The category of commutative graded rings is $\cat{CGRng}$.

A {\em commutative differential graded ring} is a commutative graded ring $A$,
equipped with a differential $\d_A$ of degree $1$, satisfying
$\d_A \circ \lmsp \d_A = 0$ and
$\d_A(a \cd b) = \d_A(a) \cd b \pl (-1)^i \cd a \cd \d_A(b)$
for all  $a \in A^i$ and $b \in A^j$.
We shall use the abbreviation {\em CDG ring} for a
commutative differential graded ring.
The category of CDG rings is denoted by $\cat{CDGRng}$.

Given a CDG ring $A$, by a {\em CDG $A$-ring} we mean a CDG ring $B$ equipped
with a homomorphism $A \to B$. We denote by $\cat{CDGRng} \ov A$ the category of
CDG $A$-rings.

The traditional name for a CDG ring $B$ is a {\em commutative associative unital
DG $\K$-algebra}, where $\K$ is some commutative base ring, e.g.\
$\K = \Z$, and the grading is often homological, i.e.\
$B = \bigoplus_{i \geq 0} B_i$.
The tacit assumption when using the expression ``$\K$-algebra'' is that the
operations in $B$ are all $\K$-linear. However, for a CDG ring $A$ and a CDG
$A$-ring $B$, the differential $\d_B$ is not $A$-linear, but only
$A^0$-linear. This is one of several reasons for preferring the name
``commutative DG ring'' over ``commutative DG algebra''.

There is a functor
$\cat{CDGRng} \to \cat{CGRng}$, $A \mapsto A^{\natural}$, which forgets the
differentials. There is also the cohomology functor
$\cat{CDGRng} \to \cat{CGRng}$,
$A \mapsto \opn{H}(A) = \boplus_{i \leq 0} \opn{H}^i(A)$.

Let $M = \bigoplus_{i \in \Z} M^i$ and $N = \bigoplus_{i \in \Z} N^i$ be DG
$A$-modules. For every $i$ let $\opn{Hom}_A(M, N)^i$
be the abelian group of degree $i$ homomorphisms
$\phi : M \to N$, such that
$\phi(a \cd m) = (-1)^{i \cd k} \cd a \cd \phi(m)$ for all $a \in A^k$
and $m \in M$, and
There is a DG $A$-module
\begin{equation} \label{eqn:135}
\opn{Hom}_A(M, N) := \boplus_{i \in \Z} \opn{Hom}_A(M, N)^i ,
\end{equation}
whose differential is
$\d(\phi) := \d_N \circ \phi - (-1)^{i} \cd \phi \circ \d_M$.
The DG $A$-modules form an $A^0$-linear DG category
$\cat{C}(A)$, with
$\opn{Hom}_{\cat{C}(A)}(M, N) := \opn{Hom}_A(M, N)$.

A {\em strict homomorphism of DG $A$-modules}
$\phi : M \to N$ is a degree $0$ homomorphism, such that
$\phi(a \cd m) = a \cd \phi(m)$ for all $a \in A$ and $m \in M$, and
$\d_N \circ \phi = \phi \circ \d_M$;
in other words, $\phi$ is a degree $0$ cocycles in $\opn{Hom}_A(M, N)$.
The category of DG $A$-modules, with strict homomorphisms, is denoted by
$\cat{C}_{\mrm{str}}(A)$. It is an $A^0$-linear abelian category.

A bounded above DG $A$-module $P = \bigoplus_{i \leq i_0} P^i$
is called {\em semi-free} if $P^{\natural}$ is a graded free
$A^{\natural}$-module.

Let $X = \coprod_{i \leq 0} X^i$ be a {\em nonpositive graded set}.
The elements of $X^i$ are called variables of degree $i$.
The {\em commutative graded polynomial ring} on $X$ is the graded ring
$\Z[X]$ generated by the set of variables $X$, modulo the strong commutativity
relations. A CDG $A$-ring $B$ is called a {\em semi-free CDG $A$-ring} if
$B^{\natural} \cong A^{\natural} \ot_{\Z} \Z[X]$
as graded $A$-rings for some nonpositive graded set $X$.

For more details on DG algebra see Chapter 3 of the book \cite{Ye3}, and the
papers \cite{Ye1} and \cite{Ye2}.

A {\em divided power structure}, abbreviated to {\em PD structure}, on a
CDG ring $A$, consists of functions
$\ga^k : A^i \to A^{k \cd i}$ for all integers $k \geq 0$ and even $i \leq -2$.
The two conditions that are important for us are these:
\begin{itemize}
\rmitem{PD1} For all even $i \leq -2$, all $a \in A^i$, and all
$k, l \geq 0$, these equalities hold:
$\ga^0(a) = 1$, $\ga^1(a) = a$, and
$\ga^k(a) \cd \ga^l(a) = \binom{k + l}{l} \cd \ga^{k + l}(a)$.

\rmitem{PD2} For all even $i \leq -2$, all $a \in A^i$, and all
$k \geq 1$, this equality holds:
 $\d_A(\ga^k(a)) = \d_A(a) \cd \ga^{k - 1}(a)$.
\end{itemize}
There are more conditions that we won't need, see \cite{AH}, \cite{Ri} or
\cite[Chapter tag =
\href{https://stacks.math.columbia.edu/tag/08P5}{\texttt 08P5}]{SP}.
Induction on $k$ shows that $a^k = (k!) \cd \ga^k(a)$.

In characteristic $0$, i.e.\ when $\Q \sub A^0$, there is a unique PD structure
on $A$, which is $\ga^k(a) := (k!)^{-1} \cd a^k$. But otherwise PD structures
need not exist; an easy counterexample is the CDG ring $A := \Z[x]$, where $x$
is a variable of degree $-2$ and $\d_A(x) = 0$.

\begin{lem} \label{lem:100}
Let $n \leq -1$ be odd. Take variables $x$ and $y$, with $\opn{deg}(x) = n$ and
$\opn{deg}(y) = n + 1$.
\begin{enumerate}
\item There is a unique differential $\d_C$ on the commutative graded polynomial
ring $\Z[x, y]$, satisfying $\d_C(x) = y$ and $\d_C(y) = 0$.
The resulting semi-free CDG $\Z$-ring, with semi-basis
$\{ x, y \}$, is denoted by $C$.

\item There is a unique CDG ring homomorphism $p : C \to \Z$ such that
$p \circ e = \opn{id}_{\Z}$, $p(x) = 0$, and $p(y) = 0$.

\item The inclusion $e : \Z \to C$ is a homotopy equivalence in
$\cat{C}_{\mrm{str}}(\Z)$.
\end{enumerate}
\end{lem}

\begin{proof}
(1) According to \cite[Lemma 3.20]{Ye2} the differential $\d_C$ exists and it
unique. Clearly $C$ is a semi-free CDG $\Z$-ring.

\medskip \noindent
(2) By \cite[Lemma 3.19]{Ye2},

\medskip \noindent
(3) As graded $\Z$-modules we have
\[ C^{\natural} = \bigoplus_{0 \leq i \leq 1, \mmsp 0 \leq j}
\Z \cd x^i \cd y^j , \]
and each $\Z \cd x^i \cd y^j$ is a rank $1$ free $\Z$-module.
Define $M_0 := \Z \cd 1_C$, which is a sub DG $\Z$-module of $C$, and
$e : \Z \to M_0$ is an isomorphism in $\cat{C}_{\mrm{str}}(\Z)$.
For $i \geq 1$ define
\[ M_i := \bigl( \Z \cd x \cd y^i \bigr) \oplus \bigl( \Z \cd y^{i + 1} \bigr)
. \]
Because $\d(x \cd y^i) = y^{i + 1}$ and
$\d(y^{i + 1}) = 0$, we see that $M_i$ is a sub DG $\Z$-module of $C$. Therefore
$C =  \bigoplus_{i \geq 0} M_i$ as DG $\Z$-modules. For every $i \geq 1$ the
DG $\Z$-module $M_i$ is contractible, so the inclusion $M_0 \to C$ is a
homotopy equivalence in $\cat{C}_{\mrm{str}}(\Z)$.
\end{proof}

\begin{lem} \label{lem:101}
Let $n \leq -2$ be even. Take variables $y, x_1, x_2, x_3, \ldots$
with $\opn{deg}(y) = n + 1$ $\opn{deg}(x_i) = n \cd i$.
Define the graded set
$X := \{ x_i \}_{i \geq 1} \cup \{ y \}$.
\begin{enumerate}
\item There is a unique differential $\d_{\til{C}}$ on the commutative graded
polynomial ring $\Z[X]$,
satisfying $\d_{\til{C}}(y) = 0$, $\d_{\til{C}}(x_1) = y$, and
$\d_{\til{C}}(x_i) = x_{i - 1} \cd y$ for $i \geq 2$.
The resulting semi-free CDG $\Z$-ring, with semi-basis
$X$, is denoted by $\til{C}$.

\item There is a unique CDG ring homomorphism $p : \til{C} \to \Z$, such that
$p \circ e = \opn{id}_{\Z}$, $p(x_i) = 0$, and $p(y) = 0$.

\item Let $\til{R} \sub \til{C}$ be the ideal generated by the elements
$r_{i, j} := x_i \cd x_j - \binom{i + j}{j} \cd x_{i + j}$, for
$i, j \geq 1$. Then $\til{R}$ is a DG ideal of $\til{C}$, and $p(\til{R}) = 0$.

\item Define the CDG ring $C := \til{C} / \til{R}$, and let $\bar{y}$ and
$\bar{x}_i$ by the images of $y$ and $x_i$ in $C$.
Then $C$ is a semi-free DG $\Z$-module, with semi-basis
$\{ \bar{x}_i \cd \bar{y}^j \}_{1 \leq i, \mmsp 0 \leq j \leq 1}$.

\item The inclusion $e : \Z \to C$ is a homotopy equivalence in
$\cat{C}_{\mrm{str}}(\Z)$.

\item There is a unique CDG ring homomorphism $p : C \to \Z$, such that
$p \circ e = \opn{id}_{\Z}$, $p(\bar{y}) = 0$, and $p(\bar{x}_i) = 0$.
\end{enumerate}
\end{lem}

\begin{proof}
(1) According to \cite[Lemma 3.20]{Ye2} the differential $\d_{\til{C}}$ exists
and it unique. Clearly $\til{C}$ is a semi-free CDG $\Z$-ring, with semi-basis
$X$.

\medskip \noindent
(2) By \cite[Lemma 3.19]{Ye2}.

\medskip \noindent
(3) Since
$\opn{deg}(x_i \cd x_j) = n \cd (i + j) = \opn{deg}(x_{i + j})$,
the ideal $\til{R}$ is graded. To prove that
$\d_{\til{C}}(\til{R}) \sub \til{R}$ it
suffices to prove that $\d_{\til{C}}(r_{i, j}) \in \til{R}$ for all
$i, j \geq 1$.
This is a calculation. Introducing $x_0 := 1$, we have
$\d_{\til{C}}(x_i) = x_{i  -1} \cd y$ for all $i \geq 1$.
Using the binomial identity
$\binom{i \mmsp + \mmsp j - 1}{j} + \binom{i \mmsp + \mmsp j - 1}{j - 1}
= \binom{i \mmsp + \mmsp j}{j}$
and the fact that $n$ is even, we obtain
\[ \begin{aligned}
&
\d_{\til{C}}(r_{i, j}) = \d_{\til{C}}(x_i) \cd x_j + (-1)^{n \cd i} \cd
x_i \cd \d_{\til{C}}(x_j) - {\textstyle \binom{i + j}{j}} \cd
\d_{\til{C}}(x_{i + j})
\\ & \quad
= x_{i - 1} \cd y \cd x_j + x_i \cd x_{j - 1} \cd y
- {\textstyle \binom{i + j}{j}} \cd x_{i + j - 1} \cd y
\\ & \quad
= \bigl( x_{i - 1} \cd x_j + x_i \cd x_{j - 1}
- {\textstyle \binom{i + j}{j}} \cd x_{i + j - 1} \bigr) \cd y
\\ & \quad
= \Bigl(
\bigl( r_{i - 1, j} + {\textstyle \binom{i \pl j - 1}{j}}
\cd x_{i \pl j - 1} \bigr)
+ \bigl( r_{i, j - 1} + {\textstyle \binom{i \pl j - 1}{j - 1}}
\cd x_{i \pl j - 1} \bigr)
- {\textstyle \binom{i + j}{j}} \cd x_{i + j - 1} \Bigr) \cd y
\\ & \quad
= (r_{i - 1, j} + r_{i, j - 1}) \cd y \in \til{R} .
\end{aligned} \]
Since $p(x_i) = 0$ for all $i \geq 1$ it follows that
$p(r_{i, j}) = 0$ for all $i, j \geq 1$, and thus $p(\til{R}) = 0$.

\medskip \noindent
(4) Since $\bar{y}$ has odd degree, we have $\bar{y}^2 = 0$.
Because $r_{i, j} \in \til{R}$, there is equality
$\bar{x}_i \cd \bar{x}_j - \binom{i + j}{j} \cd \bar{x}_{i + j} = 0$
in $C$. Therefore each $\Z \cd \bar{x}_i \cd \bar{y}^j$
is a free $\Z$-module of rank $1$, and
\[  C^{\natural} = \bigoplus_{1 \leq i, \mmsp 0 \leq j \leq 1}
\Z \cd \bar{x}_i \cd \bar{y}^j \]
as graded $\Z$-modules.
We see that $C^{\natural}$ is a bounded above graded free $\Z$-module, and
therefore $C$ is a semi-free DG $\Z$-module.

\medskip \noindent
(5) The differential $\d_C$ satisfies
$\d_C(\bar{x}_i \cd \bar{y}) = \bar{x}_{i - 1} \cd \bar{y} \cd \bar{y} = 0$
and
$\d_C(\bar{x}_i) = \bar{x}_{i - 1} \cd \bar{y}$.
Define $M_0 := \Z \cd 1_C$, and for $i \geq 1$ define
\[ M_i := \bigl( \Z \cd \bar{x}_i \bigr) \oplus
\bigl( \Z \cd \bar{x}_{i - 1} \cd \bar{y} \bigr) . \]
Then every $M_i$ is a sub DG $\Z$-module of $C$, and
$C = \bigoplus_{i \geq 0} M_i$ in $\cat{C}_{\mrm{str}}(\Z)$.
For every $i \geq 1$ the DG $\Z$-module $M_i$ is contractible, so the inclusion
$M_0 \to C$ is a homotopy equivalence in $\cat{C}_{\mrm{str}}(\Z)$.
And $e : \Z \to M_0$ is an isomorphism in
$\cat{C}_{\mrm{str}}(\Z)$.

\medskip \noindent
(6) This follows immediately from items (2)-(4).
\end{proof}

Fix a CDG ring $A$. Let $\{C_s \}_{s \in S}$ be a collection of objects of
$\cat{CDGRng} \ov A$, indexed by a set $S$. The coproduct of this collection in
$\cat{CDGRng} \ov A$ is denoted by $C_S$.
This coproduct can be realized as follows.
Let us denote by $\opn{Fin}(S)$ the set of finite subsets
$S' \sub S$, which is a directed set by inclusion.
Choose an ordering on the set $S$. Then every $S' \in \opn{Fin}(S)$ is
ordered, and every inclusion $S' \sub S''$ is order preserving.
Take some $S' \in \opn{Fin}(S)$, and write it as
$S' = \{ s_1, \ldots, s_m \}$ with $s_i < s_{i + 1}$.
Then the tensor product
$C_{S'} := C_{s_1} \ot_A \cdots \ot_A C_{s_m}$
exists in $\cat{CDGRng} \ov A$, and for an inclusion $S' \sub S''$ there is an
obvious homomorphism $C_{S'} \to C_{S''}$ in $\cat{CDGRng} \ov A$.
We obtain a direct system
$\{ C_{S'} \}_{S' \in \opn{Fin}(S)}$ in $\cat{CDGRng} \ov A$,
and $C_S \cong \lim\limits_{S' \to} C_{S'}$, the direct limit of this system.
For this reason we shall denote the coproduct $C_S$ by
$\bigotimes_{s \in S} C_s$.

Let $M$ be a DG $A$-module. Since $\d_A(A^0) = 0$, the differential $\d_M$ is
$A^0$-linear. For every integer $i$ we have the $A^0$-module
$\opn{Z}^i(M)$ of degree $i$ cocycles, and  the $A^0$-module $\opn{B}^i(M)$ of
degree $i$ coboundaries.
It is convenient to introduce the ad hoc object
$\opn{W}^i(M) := M^i / \opn{Z}^i(M)$.
In this way we have $A^0$-linear functors
\[ \opn{B}^i, \opn{Z}^i, \opn{Id}^i, \opn{W}^i, \opn{H}^i :
\cat{C}_{\mrm{str}}(A) \to \cat{M}(A^0) , \]
where $\cat{M}(A^0)$ is the category of $A^0$-modules.
There are the following morphisms of functors:
\[ \begin{tikzcd} [column sep = 6ex, row sep = 4ex]
\opn{Id}^{i - 1}
\ar[r, two heads, "{\opn{d}^{i - 1}}"]
&
\opn{B}^i
\ar[r, tail]
&
\opn{Z}^{i}
\ar[r, two heads]
\ar[d, tail]
&
\opn{H}^i
\\
&
&
\opn{Id}^{i}
\ar[r, two heads]
&
\opn{W}^{i}
\end{tikzcd} \]
And for $M \in \cat{C}_{\mrm{str}}(A)$ there are functorial short exact
sequences in $\cat{M}(A^0)$~:
\begin{equation} \label{eqn:115}
\begin{split}
0 \to \opn{Z}^{i - 1}(M) \to & \mmsp M^{i - 1} \xar{\mmsp \d^{i - 1}_M \mmsp}
\opn{B}^{i}(M) \to 0  \, ,
\\
0 \to \opn{B}^i(M) \to & \mmsp \opn{Z}^i(M) \to \opn{H}^i(M) \to 0 \, ,
\\
0 \to \opn{Z}^i(M) \to & \mmsp M^i \to \opn{W}^{i}(M) \to 0 \, .
\end{split}
\end{equation}

\medskip
\begin{proof}[Proof of Theorem \tup{\ref{thm:100}}]
The proof is broken up into a few steps.

\medskip \noindent
Step 1. For each integer $n \leq -1$ choose a collection
$\{ b_s \}_{s \in S^n}$ of elements of $B^n$, indexed by a set $S^n$,
whose images in $\opn{W}^n(B) = B^n / \opn{Z}^n(B)$ generate it as an
$A^0$-module. Define the graded set $S := \coprod_{n \leq -1} S^n$.

\medskip \noindent
Step 2.
Take an odd integer $n \leq -1$.

For an element $s \in S^n$ let $C_s$ be a copy
of the CDG ring $C$ from Lemma \ref{lem:100}, with
generators $x \in C_s^n$ and $y \in C_s^{n + 1}$.
According to \cite[Lemma 3.19]{Ye2} there exists a unique CDG $\Z$-ring
homomorphism $g_s : C_s \to B$ such that
$g_s(x) = b_s$ and $g_s(y) = \d_B(b_s)$.
There are also unique CDG ring homomorphisms
$e_s : \Z \to C_s$ and $p_s : C_s \to \Z$,
such that $p_s(x) = 0$, $p_s(y) = 0$, and $p_s \circ e_s = \opn{id}_{\Z}$.
For the existence and uniqueness of $p_s$ we rely on \cite[Lemma 3.19]{Ye2}.
By Lemma \ref{lem:100}(3) the inclusion $e_s : \Z \to C_s$ is a homotopy
equivalence in $\cat{C}_{\mrm{str}}(\Z)$.

Define the CDG ring
$C_{n} := \bigotimes_{s \in S^n} C_s$.
By the functoriality of the coproduct,
we get induced CDG ring homomorphisms
$g_n : C_n \to B$, $e_n : \Z \to C_n$ and $p_n : C_n \to \Z$,
such that $p_n \circ e_n = \opn{id}_{\Z}$.
For every $S' \in \opn{Fin}(S^n)$ let
$C_{S'} := \bigotimes_{s \in S'} C_s$; so
the inclusion $e_{S'} : \Z \to C_{Z'}$ is a homotopy
equivalence in $\cat{C}_{\mrm{str}}(\Z)$.
Because $C_n \cong \lim\limits_{S' \to} C_{S'}$,
it follows that the inclusion $e_n : \Z \to C_n$ is a homotopy equivalence in
$\cat{C}_{\mrm{str}}(\Z)$.

\medskip \noindent
Step 3. Choose a PD structure $\{ \ga^k \}_{k \geq 0}$ on the CDG ring $B$.
Take an even integer $n \leq -2$.

For an element $s \in S^n$ let $\til{C}_s$ be a copy of the CDG ring $\til{C}$
from Lemma \ref{lem:101}.
It is a semi-free CDG $\Z$-ring with semi-basis
$X = \{ x_i \}_{i \geq 1} \cup \{ y \}$.
According to \cite[Lemma 3.19]{Ye2} and condition (PD2) there exists a unique
CDG $\Z$-ring
homomorphism $\til{g}_s : \til{C}_s \to B$, such that
$\til{g}_s(x_i) = \ga^i(b_s)$ and $\til{g}_s(y) = \d_B(b_s)$.
There are also unique CDG ring homomorphisms
$\til{e}_s : \Z \to \til{C}_s$ and $\til{p}_s : \til{C}_s \to \Z$,
such that $\til{p}_s(x) = 0$, $\til{p}_s(y) = 0$, and
$\til{p}_s \circ \til{e}_s = \opn{id}_{\Z}$.

Let $r_{i, j}$ be the relations from Lemma \ref{lem:101}(3), and let
$\til{R} \sub \til{C}_s$ be the ideal generated by the collection
$\{ r_{i, j} \}_{i, j \geq 1}$.
Define the CDG ring $C_s := \til{C}_s / \til{R}$.
By condition (PD1), for every $i, j \geq 1$ we have
\[ \til{g}_s(r_{i, j}) =
\ga^i(b_s) \cd \ga^j(b_s) - {\textstyle \binom{i + j}{j}} \cd
\ga^{i + j}(b_s) = 0 \]
in $B$. Therefore $\til{g}_s(\til{R}) = 0$, and there is an induced CDG ring
homomorphism $g_s : C_s \to B$.
There are also induced CDG ring homomorphisms
$e_s : \Z \to C_s$ and $p_s : C_s \to \Z$,
such that $p_s \circ e_s = \opn{id}_{\Z}$.
By Lemma \ref{lem:101}(5) the inclusion $e_s : \Z \to C_s$ is a homotopy
equivalence in $\cat{C}_{\mrm{str}}(\Z)$.

Define the CDG ring
$C_{n} := \bigotimes_{s \in S^n} C_s$.
We get induced CDG ring homomorphisms
$g_n : C_n \to B$, $e_n : \Z \to C_n$ and $p_n : C_n \to \Z$,
such that $p_n \circ e_n = \opn{id}_{\Z}$.
The inclusion $e_n : \Z \to C_n$ is a homotopy equivalence in
$\cat{C}_{\mrm{str}}(\Z)$.

\medskip \noindent
Step 4.
Define the CDG ring
$C := \bigotimes_{n \leq -1} C_n$.
There are induced CDG ring homomorphisms
$g_C : C \to B$, $e_C : \Z \to C$ and $p_C : C \to \Z$,
such that $p_C \circ e_C = \opn{id}_{\Z}$.
The inclusion $e_C : \Z \to C$ is a homotopy equivalence in
$\cat{C}_{\mrm{str}}(\Z)$.
Here is the commutative diagram in $\cat{CDGRng}$ describing this step:
\begin{equation} \label{eqn:119}
\begin{tikzcd} [column sep = 16ex, row sep = 6ex]
\Z
&
C
\ar[l, two heads, "{p_C, \mmsp \mrm{sqi}}"']
\ar[d, "{g_C}"]
\\
\Z
\ar[u, two heads, tail, "{\opn{id}_{\Z}, \mmsp \mrm{isom}}"]
\ar[ur, tail, "{e_C, \mmsp \mrm{qi}}"]
\ar[r]
&
B
\end{tikzcd}
\end{equation}

\medskip \noindent
Step 5. Define the CDG ring
$\til{B} := A \ot_{\Z} C$.
There are induced CDG ring homomorphisms
$\til{f} := f \ot g_C : \til{B} \to B$,
$e := \opn{id}_A \ot \mmsp e_C :  A \to \til{B}$, and
$p := \opn{id}_A \ot \mmsp p_C :  \til{B} \to B$,
such that $p \circ e = \opn{id}_{A}$.
See diagram (\ref{eqn:100}).

Since $e_C : \Z \to C$ is a homotopy equivalence in
$\cat{C}_{\mrm{str}}(\Z)$, it follows that $e : A \to \til{B}$ is a homotopy
equivalence in $\cat{C}_{\mrm{str}}(A)$. Hence $e : A \to \til{B}$
it is a quasi-isomorphism in $\cat{CDGRng}$.
A diagram chase, in diagram (\ref{eqn:100}), shows that
$p : \til{B} \to A$ and $\til{f} : \til{B} \to B$ are
quasi-isomorphisms too.

\medskip \noindent
Step 6. It remains to prove that $\til{f} : \til{B} \to B$ is surjective.

By construction (in step 1), for every $i \leq -1$ the $A^0$-module
homomorphism
$\phi_i : \lb A^0 \ot_{\Z} C^i \to \opn{W}^i(B)$
induced by the inclusion $A^0 \ot_{\Z} C^i \to B^i$
is surjective. Looking at the commutative diagram
\[ \begin{tikzcd} [column sep = 7ex, row sep = 6ex]
A^0 \ot_{\Z} C^i
\ar[r]
\arrow[rrr, two heads, bend left = 20, "{\phi_i}"]
&
\til{B}^i = (A \ot_{\Z} C)^i
\arrow[r]
&
\opn{W}^i(\til{B})
\arrow[r, "{\opn{W}^i(\til{f})}"']
&
\opn{W}^i(B)
\end{tikzcd} \]
we see that $\opn{W}^i(\til{f}) : \opn{W}^i(\til{B}) \to \opn{W}^i(B)$
is surjective.
The homomorphism $\opn{W}^0(\til{f}) : \opn{W}^0(\til{B}) \to \opn{W}^0(B)$
is surjective because $\opn{W}^0(B) = 0$.
Thus $\opn{W}^0(\til{f})$ is surjective for all $i \leq 0$.

For every $i \leq 0$ the
surjection $\d_B^{i - 1} : B^{i - 1} \to \opn{B}^{i}(B)$
induces a surjection
$\opn{W}^{i - 1}(B) \to \opn{B}^{i}(B)$.
The commutative diagram
\[ \begin{tikzcd} [column sep = 8ex, row sep = 6ex]
\opn{W}^{i - 1}(\til{B})
\ar[r, two heads]
\ar[d, two heads, "{\opn{W}^{i - 1}(\til{f})}"']
&
\opn{B}^{i}(\til{B})
\ar[d, "{\opn{B}^{i}(\til{f})}"]
\\
\opn{W}^{i - 1}(B)
\ar[r, two heads]
&
\opn{B}^{i}(B)
\end{tikzcd} \]
in $\cat{M}(A^0)$, with the surjectivity of $\opn{W}^{i - 1}(\til{f})$, show
that $\opn{B}^i(\til{f}) : \opn{B}^i(\til{A}) \to \opn{B}^i(B)$ is surjective.

For every $i \leq 0$ we have this commutative diagram
\[ \begin{tikzcd} [column sep = 6ex, row sep = 6ex]
0
\ar[r]
&
\opn{B}^{i}(\til{B})
\ar[r]
\ar[d, two heads, "{\opn{B}^{i}(\til{f})}"']
&
\opn{Z}^{i}(\til{B})
\ar[r]
\ar[d, "{\opn{Z}^{i}(\til{f})}"']
&
\opn{H}^{i}(\til{B})
\ar[r]
\ar[d, two heads, tail, "{\opn{H}^{i}(\til{f})}"']
&
0
\\
0
\ar[r]
&
\opn{B}^{i}(B)
\ar[r]
&
\opn{Z}^{i}(B)
\ar[r]
&
\opn{H}^{i}(B)
\ar[r]
&
0
\end{tikzcd} \]
in $\cat{M}(A^0)$ with exact rows, in which $\opn{B}^{i}(\til{f})$ is
surjective (as we proved above),  and $\opn{H}^{i}(\til{f})$ is bijective
(because $\til{f}$ is a quasi-isomorphism). Therefore
$\opn{Z}^i(\til{f}) : \opn{Z}^i(\til{B}) \to \opn{Z}^i(B)$ is surjective.

Lastly, for every $i \leq 0$ we have this commutative diagram
\[ \begin{tikzcd} [column sep = 6ex, row sep = 6ex]
0
\ar[r]
&
\opn{Z}^{i}(\til{B})
\ar[r]
\ar[d, two heads, "{\opn{Z}^{i}(\til{f})}"']
&
\til{B}^i
\ar[r]
\ar[d, "{\til{f}^i}"']
&
\opn{W}^{i}(\til{B})
\ar[r]
\ar[d, two heads, tail, "{\opn{W}^{i}(\til{f})}"']
&
0
\\
0
\ar[r]
&
\opn{Z}^{i}(B)
\ar[r]
&
B^{i}
\ar[r]
&
\opn{W}^{i}(B)
\ar[r]
&
0
\end{tikzcd} \]
in $\cat{M}(A^0)$ with exact rows, in which $\opn{Z}^{i}(\til{f})$ and
$\opn{W}^{i}(\til{f})$ are surjective.
Therefore
$\til{f}^i : \til{A}^i \to B^i$ is surjective.
Summing on all $i \leq 0$ we conclude that
$\til{f} : \til{A} \to B$ is surjective.
\end{proof}

Here are two remarks to end the paper.

\begin{rem} \label{rem:120}
The proof of Theorem \ref{thm:100} gives a little bit more: the CDG $A$-ring
$\til{B} = A \ot_{\Z} C$ constructed there is {\em semi-free as a DG
$A$-module}.

But $\til{B}$ is {\em not semi-free as a CDG $A$-ring}.
To see why, consider the CDG ring $C$ in Lemma \ref{lem:101}, for an even
integer $n \leq -2$. The element $\bar{x}_1 \in C$ satisfies
$(\bar{x}_1)^2 = 2 \cd \bar{x}_2$, etc. Thus the collection
$\{ \bar{x}_i \}_{i \geq 1}$ of elements of $C$ is {\em linearly independent}
over $\Z$, but it is {\em not algebraically independent}.

The construction of the CDG ring $\til{B}$ is a generalization of the {\em Tate
resolution}, see
\cite[Lemmas
tag = \href{https://stacks.math.columbia.edu/tag/09PN}{\texttt 09PN}
and tag = \href{https://stacks.math.columbia.edu/tag/09PP}{\texttt 09PP}]{SP}.
\end{rem}

\begin{rem} \label{rem:140}
Theorem \ref{thm:100} will be used in our upcoming papers \cite{Ye4} and
\cite{Ye5}. The overall project is outlined in our lecture notes \cite{Ye6}.
However, we should warn the reader that at present {\em the proof of Theorem
\tup{5.1} in \cite{Ye6} is broken}, and hence the subsequent statements in
\cite{Ye6} should be taken with as tentative only.
\end{rem}


\end{document}